\documentclass[12pt]{amsart}
\pdfoutput=1
\usepackage{amssymb}
\usepackage{amsmath}
\usepackage{amsfonts}
\usepackage[usenames]{color}
\usepackage{graphicx}
\usepackage{array}
\usepackage{psfrag}
\usepackage{color}
\usepackage{ulem}
\usepackage{fancyhdr}
\usepackage{hyperref}

\makeatletter
 
 \@addtoreset{equation}{section}
\makeatother

\newtheorem{theorem}{Theorem}

\newtheorem{proposition}[theorem]{Proposition}

\theoremstyle{definition}

\theoremstyle{remark}
\newtheorem{remark}[theorem]{Remark}

\numberwithin{equation}{section}

\newcommand{\Z}{\mathbb{Z}}
\newcommand{\R}{\mathbb{R}}

\newcommand{\dfn}[1]{\textit{#1}}

\newcommand{\lk}{\mathrm{lk}}

\begin{document}

\title{The complement problem for linklessly embeddable graphs}

\date{\today}

\author{Ramin Naimi}
\address{Occidental College, Los Angeles, 90041 CA}
\email{rnaimi@oxy.edu}

\author{Ryan Odeneal}
\address{University of Nevada, Reno, Reno, NV 89557}
\email{rodeneal@unr.edu}

\author{Andrei Pavelescu}
\address{University of South Alabama, Mobile, AL 36688}
\email{andreipavelescu@southalabama.edu}

\author{Elena Pavelescu}
\address{University of South Alabama, Mobile, AL 36688}
\email{elenapavelescu@southalabama.edu}

\maketitle

\begin{abstract}
We find all maximal linklessly embeddable graphs of order up to 11,
and verify that for every graph $G$ of order 11 
either $G$ or its complement $cG$ is intrinsically linked. 
We give an example of a  graph $G$ of order 11 such that both $G$ and $cG$ are $K_6$-minor free. 
We provide minimal order examples of maximal linklessly embeddable graphs 
that are not triangular or not 3-connected. 
We prove a Nordhaus-Gaddum type conjecture on the Colin de Verdi\`ere invariant
for graphs on at most 11 vertices.
We give a description of the programs used in the search.
\end{abstract}

\section{Introduction}

Throughout this article, we work with simple graphs.
We continue the study of graphs and their complements initiated nearly sixty years ago by Battle, Harary and Kodama~\cite{BHK}.
They established that the complement of a planar graph with nine vertices is necessarily non-planar.  
This result was independently proved by Tutte in \cite{Tutte}.
We express this by saying that the complete graph $K_9$ is not bi-planar.
In a series of articles, Harary and  Akiyama  investigated conditions under which both a graph $G$ and its complement $cG$ possess a specified property: have  connectivity one, have line-connectivity one, are 2-connected, are forests, are bipartite, are outerplanar, are Eulerian \cite{AH1}, have the same girth, have circumference 3 or 4 \cite{AH3}, have the same number of endpoints \cite{AH4}.
Given a property $\mathcal{P}$ for graphs,  
we say that the complete graph $K_n$ is \dfn{bi-$\mathcal{P}$} 
if there exists a subgraph $G$ of $K_n$ such that both $G$ and its complement $cG$
are $P$. 

Ichihara and Mattman \cite{IM} proved that the complement of an $n$-apex graph with $2n+9$ vertices is not $n$-apex, that is, $K_{2n+9}$ is not bi-$n$-apex.
In particular  $K_{11}$ is not bi-1-apex. 
A graph is \dfn{$n$-apex} if deleting some set of $n$ vertices yields a planar graph.
A 1-apex graph is usually called apex.

This article is about the property of linkless embedability.
A graph is \dfn{intrinsically linked} (\dfn{IL}) if every embedding of it in $\R^3$ (or $S^3$) contains a nonsplittable 2-component link.
A graph is \dfn{linklessly embeddable} if it is not intrinsically linked (\dfn{nIL}).
The combined work of Conway and Gordon \cite{CG}, Sachs \cite{Sa}, and Robertson, Seymour and Thomas \cite{RST} fully characterize IL graphs: a graph is IL if and only if it contains a graph in the Petersen family as a minor.
The Petersen family consists of seven graphs obtained from $K_6$ by $\nabla Y-$moves and $Y\nabla-$moves, as presented in Figure \ref{fig-ty}.

\begin{figure}[htpb!]
\begin{center}
\begin{picture}(160, 50)
\put(0,0){\includegraphics[width=2.4in]{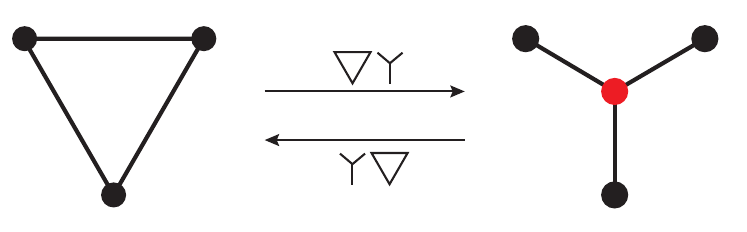}}
\end{picture}
\caption{\small  $\nabla Y-$ and $Y\nabla-$moves} 
\label{fig-ty}
\end{center}
\end{figure}

We ask: what is the smallest  $N$ such that the complete graph $K_N$ is not  bi-nIL?
It is easy to prove that $N \le 15$:
Mader \cite{Ma} showed that a graph on $n\ge 6$ vertices and at least  $4n-9$ edges contains a $K_6$ minor.
Since, for $n \ge 15$, $K_n$ has  ${n}\choose{2}$ $\ge  2(4n-9)$ edges,
for any graph $G$ with at least 15 vertices, 
either $G$ or $cG$ contains a $K_6$ minor and hence is IL.
In  \cite{PP2}, the last two authors proved $K_{10}$ is bi-apex and therefore bi-nIL (by \cite{Sa}), and that $K_{13}$ is not bi-nIL, thereby showing that $11\le N \le 13$,
for if $K_n$ is not bi-nIL, then neither is $K_{n+1}$.
In this paper we show that $N=11$.
Then we prove:
\begin{theorem}
If $G$ is a linklessly embeddable graph with at least 11 vertices, then the complement of $G$ is intrinsically linked.
\label{main}
\end{theorem}

 Kotlov, Lov\'asz and Vempala \cite{KLV} conjectured that
 for any graph of order $n$, $\mu(G)+\mu(cG)\ge n-2$,
 where $\mu$ denotes the de Verdi\`ere invariant of a graph.
 We use Theorem \ref{main}  in the next section to prove this conjecture for graphs of order 11.

A nIL graph is \textit{maxnIL} if it is not a proper subgraph of a nIL graph of the same order.
To prove Theorem~\ref{main}, it is enough to check the statement for maxnIL graphs.
This is because if $G$ is a subgraph of a nIL graph $G'$ of the same order,  and $cG'$ is IL, then $cG$ is also IL, as it contains $cG'$ as a subgraph.
To this end, using two computer algorithms that will be detailed in Section \ref{CS},  we found a complete list of maxnIL graphs of order up to 11. 
This list is available in \cite{ONPParxiv}.
For each maxnIL graph of order 11, its complement was found to be IL.
Techniques and considerations used to narrow down the computer search are presented in Section \ref{maxn}.\\


\section{maxnIL graphs of order up to 11}
\label{maxn}

By Mader \cite{Ma}, any graph of order 11 and size greater than 34 contains a $K_6$ minor, so it is IL.
On the other hand, by work of Aires \cite{A}, we know that a maxnIL graph on 11 vertices has at least 22 edges.
We therefore narrowed our search to graphs on 11 vertices and $22\le m\le 34$ edges. 
Similar bounds on the number of edges were set for $n=7,8,9,10$
(in order to find all maxnIL graphs on 11 vertices, 
we first needed to find all maxnIL graphs on 10 or fewer vertices,
as explained further below).

By  \cite{NPP}, maxnIL graphs are 2-connected; so the search was restricted to graphs with connectivity at least 2.
We further narrowed the search space by proving a set of results on maxnIL graphs,
which we describe below.\\

To \textit{cone} a vertex over a graph $H$ means we add a new vertex  $v$ and connect $v$ to every vertex of $H$.
Sachs \cite{Sa} showed that a graph $H$ is planar if and only if coning one vertex over $H$ yields a nIL graph.
It follows that
if a nIL graph $G$ of order $n$ has a vertex $v$ of degree $n-1$, 
then $G\setminus v$ is planar. 
We use this to show:

\begin{proposition}
\label{prop-maxnil-apex}
Let $G$ be a maxnIL graph of order $n$.
Then the following are equivalent: 
(i)~$G$ is apex;
(ii)~$G$ has a vertex of degree $n-1$;
(iii)~$G$ is a cone on a maximal planar graph of order $n-1$.
\end{proposition}

\begin{proof} 
(i) $\to$ (ii):
Suppose $G$ is apex.
Then for some vertex $v \in G$, $G\setminus v$ is planar.
If $\deg(v) < n-1$, then an edge $e$ incident to $v$ can be added to $G$ 
such that $G+e$ is nIL, contradicting the maximality of $G$.
(ii) $\to$ (iii):
If a vertex $v$ has degree $n-1$, 
then, by \cite{Sa}, $G\setminus v$ is planar.
It must be maximal planar since otherwise
an edge $e$ can be added to  $G\setminus v$ while preserving its planarity,
so that $G+e$ is nIL, again a contradiction.
(iii) $\to$ (i): This follows from the definition of apex.
\end{proof}
 
This means that each plane triangulation of order $n\ge 5$ gives rise to a unique  maxnIL graph of order $n+1$. 
There is only one plane triangulation with 5 vertices; and, by work of Bowen and Fisk \cite{BF}, the plane triangulations with $6\le n\le 10$ vertices are also known, and their numbers are: $T_6=2, T_7=5, T_8=14, T_9=50, T_{10}=233$.
The number of maxnIL apex graphs is presented in Table 1. 

An edge $e$ of a graph $G$ is \textit{triangular} if $e$ is in some 3-cycle in $G$.
A graph is \textit{triangular}  if every edge of it is triangular.
A graph $G$ is the \dfn{clique sum} of two graphs $G_1$ and $G_2$ over $K_p$ if $V(G)=V(G_1)\cup V(G_2)$, $E(G)=E(G_1)\cup E(G_2)$, and the subgraphs induced by $V(G_1)\cap V(G_2)$ in $G_1$ and $G_2$ are   both complete of order $p$. 
Clique sum constructions are often used to create larger order graphs with a certain property (e.g., linklessly embeddable, or $K_6$-minor-free) from two graphs with the same property.

\begin{proposition}
If every maxnIL graph of order at least 3 and at most $n$ is triangular, then every maxnIL graph $G$ of order $n+1$ is 3-connected. 
\label{3-connected}
\end{proposition}
\begin{proof}
Assume $G$ is a maxnIL graph of order $n+1$ that is not 3-connected.
By \cite{NPP}, $G$ is a clique sum of two maxnIL graphs of order at most $n$
over an edge which is non-triangular in at least one of the two graphs. 
This contradicts the assumption that every maxnIL graph of order at most $n$ is triangular.
\end{proof}

The only maxnIL graphs of order $3 \le n \le 5$ are $K_n$.
And $K_6^-$ ($K_6$ minus an edge) is the only maxnIL graph of order 6.
These graphs are all triangular.
It follows by Proposition \ref{3-connected} that every maxnIL graph of order 7 is 3-connected and hence has minimal degree at least 3.
For each $n=7, 8, 9, 10$, we used the Nauty program designed by McKay~\cite{MP} to generate the 2-connected  graphs of minimal degree 3 with $n$ vertices and  $2n\le m\le 4n-10$ edges.
 We then used the Mathematica program detailed in Section \ref{CS} 
 to find the maxnIL graphs among them and to confirm they were all triangular. 
 The search considered a total of 5,065,328 graphs. 
By Proposition \ref{3-connected} , this implies that all maxnIL graphs of order 11 are 3-connected.
(However, there exists a maxnIL graph of order 11 which is not triangular.
 See Figure \ref{1127})

As our next step in finding all maxnIL graphs of order 11, 
we first focused on those with a vertex of degree 3. 
By \cite{NPP}, if a vertex $v$ of a maxnIL graph $G$  has degree 3, 
then $v$ belongs to a $K_4$ subgraph of $G$ and $G\setminus v$ is maxnIL. 
Having found all maxnIL graphs of order 10 (there are 107 of them), 
we constructed, from each such graph $G$, graphs of order 11 by
adding a new vertex and connecting it to three vertices which form a triangle in $G$.
We obtained 1963 graphs this way, of which 159 were maxnIL.
Since plane triangulations have minimal degree 3, 
by Proposition \ref{prop-maxnil-apex} maxnIL apex graphs have minimal degree 4 for $n\ge 5$.
Therefore this set of 159 maxnIL graphs consists of non-apex graphs only
and is disjoint from the set of 233 maxnIL apex graphs found above.

The search for all maxnIL graphs of order 11 was now reduced to considering 3-connected graphs
with minimum degree $\ge 4$.
Furthermore, since we had already found all maxnIL graphs of order 11 that are apex,
we now only had to search for non-apex maxnIL graphs of order 11;
by Proposition \ref{prop-maxnil-apex},  
these graphs all have maximum degree $\le 9$.

Unfortunately, (to our knowledge) Nauty does not accept connectivity 3 as a parameter,
so we had to consider all 2-connected graphs of order 11 with vertex degrees between 4 and 9.
There are 158,056,639 such graphs, which we saved in a .g6 master file.
We used the Alabama Super Computer (ASC) and the Python code detailed in Section \ref{CS} to sift out the graphs which had $K_6$ minors. 
To do this, we split the master file into subfiles with approximately 1,000,000 graphs each, and then each subfile was tested on 16 cores for up to 150 hours, with at most 15 jobs running at a time. 
Due to the nature of the algorithm, the larger graphs (size 34 for instance) ran faster, since they were more likely to contain a $K_6$ minor and be dismissed by the program. 
Some instances for size 27 timed out, so we had to run them again on 32 cores with more time allotment. This process took roughly 3 months.

Once a job would finish and a $K_6$-minor free list of graphs was produced, the list would be portioned into batches of 300,000 graphs and the Mathematica program for finding maxnIL graphs would run on up to 4 Intel Core i5-6500 @ 3.20GHz processors at a time, at a rate of roughly 120,000 graphs per day. 
This time, the jobs involving smaller size graphs would finish first, since less cycles need to be considered. 

The search yielded 318 order 11, 3-connected maxnIL graphs with vertex degrees between 4 and 9. 
All of them are necessarily non-apex (as explained above).
Together with the 159 maxnIL graphs with minimum vertex degree 3 and the 233 maxnIL apex graphs, the search produced all the 710 maxnIL graphs of order 11. 
The Mathematica program checked all their complements  to be IL.
This proves Theorem~\ref{main},
as any nIL graph of order greater than 11 contains a nIL induced subgraph of order 11.

Table~1 shows the number of maxnIL graphs for each order of at most 11.
For order 11, the search found maxnIL graphs with $m$ edges for every $27 \le m\le 34$,
and none for any other values of $m$.

\begin{table}[htpb!]
\label{TableMN}
\centering
\begin{tabular}{|c|c|c|c|c|c|c|}
\hline
   order $n$& 6 & 7 & 8 & 9 & 10 & 11   \\ \hline
  $1-$apex maxnIL graphs   & 1 & 2 & 5 & 14 & 50 & 233   \\ \hline
      all maxnIL graphs   & 1 & 2 & 6 & 20 & 107 & 710  \\ \hline
\end{tabular}\\\vspace*{0.2in}
    \caption{ maxnIL graphs with $6\le n\le 11$ vertices.}
\end{table}

\begin{remark} The graph $G_{11,27}$, depicted in  Figure \ref{1127}, has some unique properties among the maxnIL graphs of order 11. It is the only graph in that list which has 27 edges. It is also the only maxnIL graph of order 11 which is not triangular. In this graph, the edges $v_1v_8$ and $v_2v_{10}$ are non-triangular. Given that all the maxnIL graphs of order at most 10 are all triangular, $G_{11,27}$ is the minimal order example for non-triangular maxIL graphs.
\end{remark}


\begin{figure}[htpb!]
\begin{center}
\begin{picture}(250, 200)
\put(0,0){\includegraphics[width=3in]{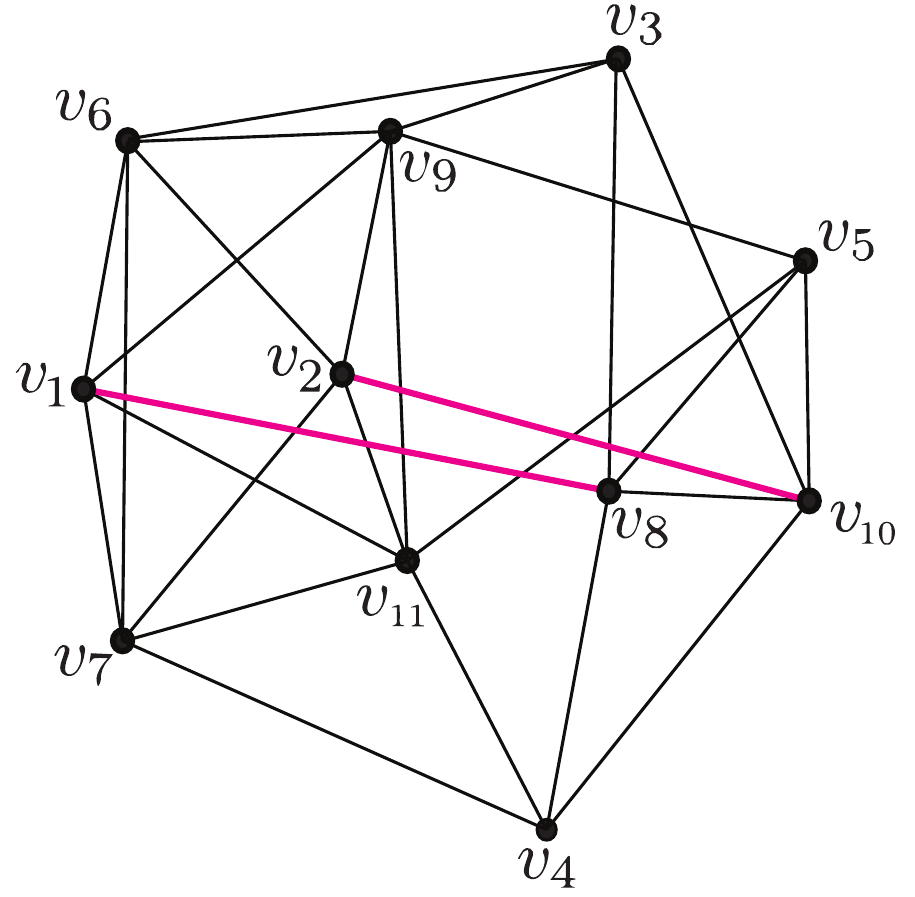}}
\end{picture}
\caption{\small MaxnIL graph with $n=11$ vertices and $m=27$ edges. Edges $v_1v_8$ and $v_2v_{10}$ are non-triangular.}
\label{1127}
\end{center}
\end{figure}

\begin{remark}
Given that all maxnIL graphs of order $\le 11$ are 3-connected,
it is natural to ask whether every maxnIL graph (of any order) is 3-connected.
We see as follows that the answer is negative.
Starting with $G_{11,27}$, adding a new vertex and connecting it to the endpoints of one of its non-triangular edges, yields a nIL graph of order 12 that is not 3-connected.
This graph is maxnIL, since it is the clique sum of two maxnIL graphs over an edge which is non-triangular in one of them \cite{NPP}. This a is a minimal order maxnIL graph with vertex connectivity 2.

\end{remark}

\begin{remark}
In light of Theorem~\ref{main},
one may wonder whether for every graph $G$ of order 11 either $G$ or $cG$ has a $K_6$ minor.
The graph in Figure \ref{1121}(a) is a maxnIL graph of order 11 since it is a cone over a maximal planar graph. 
In particular, it has no $K_6$ minor.
It turns out that the complement of this graph has no $K_6$ minor either.
Thus $K_{11}$ is ``bi-$K_6$-minor-free.''
We leave it as an open question whether $K_{12}$ is bi-$K_6$-minor-free.
%


\end{remark}


\begin{figure}[htpb!]
\begin{center}
\begin{picture}(370, 170)
\put(0,0){\includegraphics[width=5in]{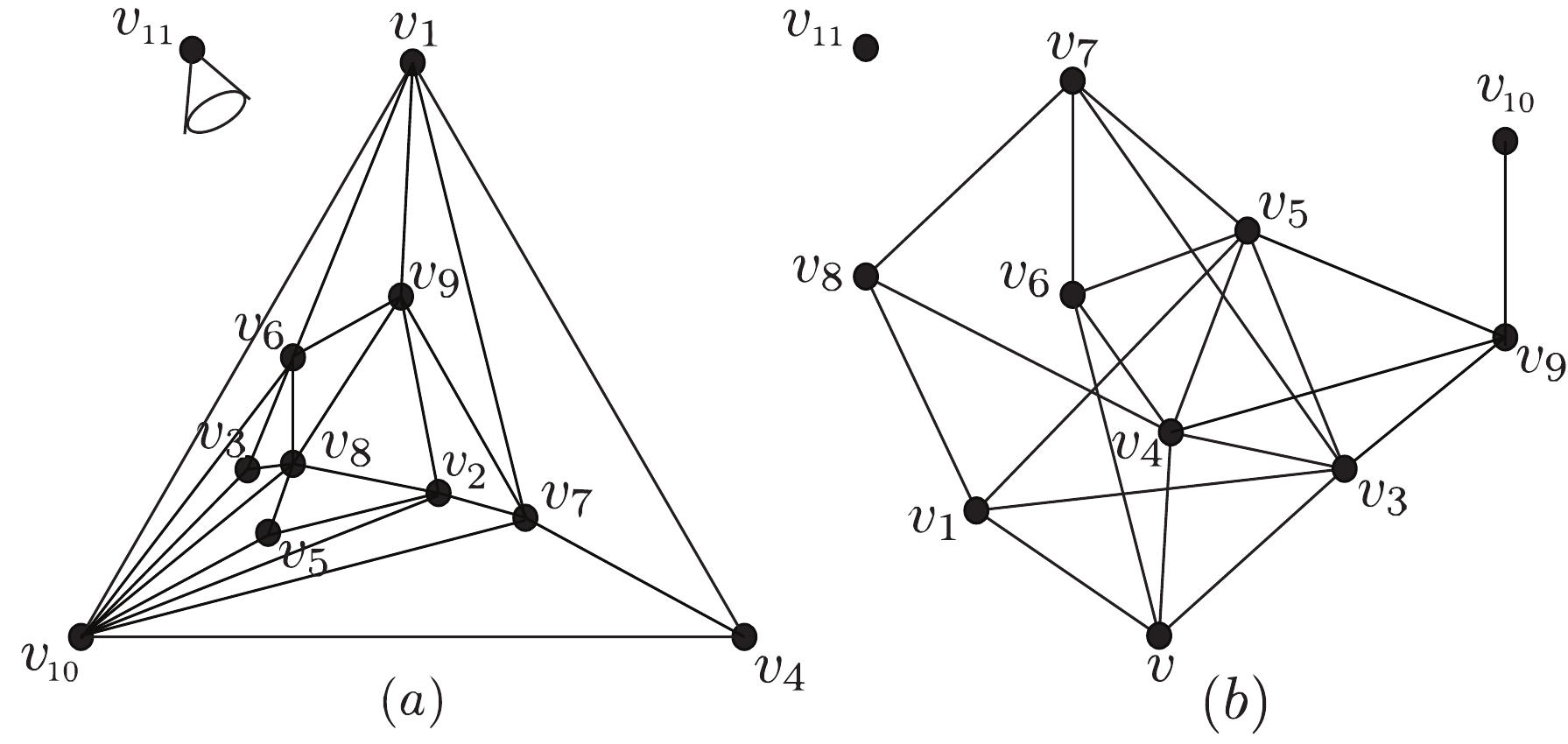}}
\end{picture}
\caption{\small (a) MaxnIL apex graph $G$. Vertex $v_{11}$ connects with all other ten vertices. (b) The complement of $G$.}
\label{1121}
\end{center}
\end{figure}

\begin{remark} The search found all maxnIL graphs of order at most 11 to be 2-apex.  
While there are known maxnIL graphs of order 13 which are not 2-apex (such as Maharry's $Q_{13,3}$ graph \cite{M}), it is still an open question whether all maxnIL graphs of order 12 are 2-apex.
\end{remark}

\begin{remark} In 1990, Colin de Verdi\'ere \cite{dV}  introduced the graph invariant $\mu$, based on spectral properties of matrices associated with a graph $G$. He showed that $\mu$  is monotone under taking minors
and that planarity is characterized by the inequality $\mu\le 3$.  Lov\'asz and Schrijver
showed that linkless embeddability is characterized by the inequality $\mu\le 4$ \cite{LS}.
By reformulating the definition of $\mu$ in terms of vector labelings, Kotlov, Lov\'asz, and Vempala \cite{KLV} related some topological properties of a graph to the $\mu$ invariant of its complement: 
for $G$ a simple graph on $n$ vertices, 
(a)~if $G$ is planar then $\mu(cG)\ge n-5$;
(b)~if $G$ is outerplanar then $\mu(cG)\ge n-4$. 
The three authors also conjectured that for any simple graph $G$ of order $n$, $\mu(G)+\mu(cG) \ge n-2$. 
The conjecture is known to hold for planar graphs, graphs of order at most 10, and graphs with $\mu(G)\ge n-6$ \cite{Hog}. 
Theorem \ref{main} validates the conjecture for graphs of order 11:
$G$ or $cG$, say $G$, is IL, so $\mu(G) \ge 5 = 11-6$;
and the conjecture holds for graphs with $\mu(G) \ge n-6$.
\end{remark}

We make a few observations about the graphs listed in Table 1.
\begin{enumerate}
\item There is exactly one non-apex maxnIL graph with 8 vertices, called the J{\o}rgensen graph \cite{J}.
The J{\o}rgensen graph is minor minimal non-apex.
\item Since the complement of each plane triangulation with 9 vertices is non-planar by \cite{BHK},  the complement of each apex maxnIL graph with 10 vertices is non-planar.
Among 50  such apex maxnil graphs, only one has a complement which is IL, with all other complements being nIL.
This complement is the graph $G_9$ in the Petersen family plus one isolated vertex.
\item
It is straightforward to show that
if a maxnIL graph of order $n$ is apex then it has $4n-10$ edges.
But the converse does not hold; 
there are maxnIL graphs of order $n$ with $4n-10$ edges that are not apex.
In Figure~\ref{non-apex} we provide a minimal order example of such a graph.

\end{enumerate}


\begin{figure}[htpb!]
\begin{center}
\begin{picture}(250, 130)
\put(0,0){\includegraphics[width=3.5in]{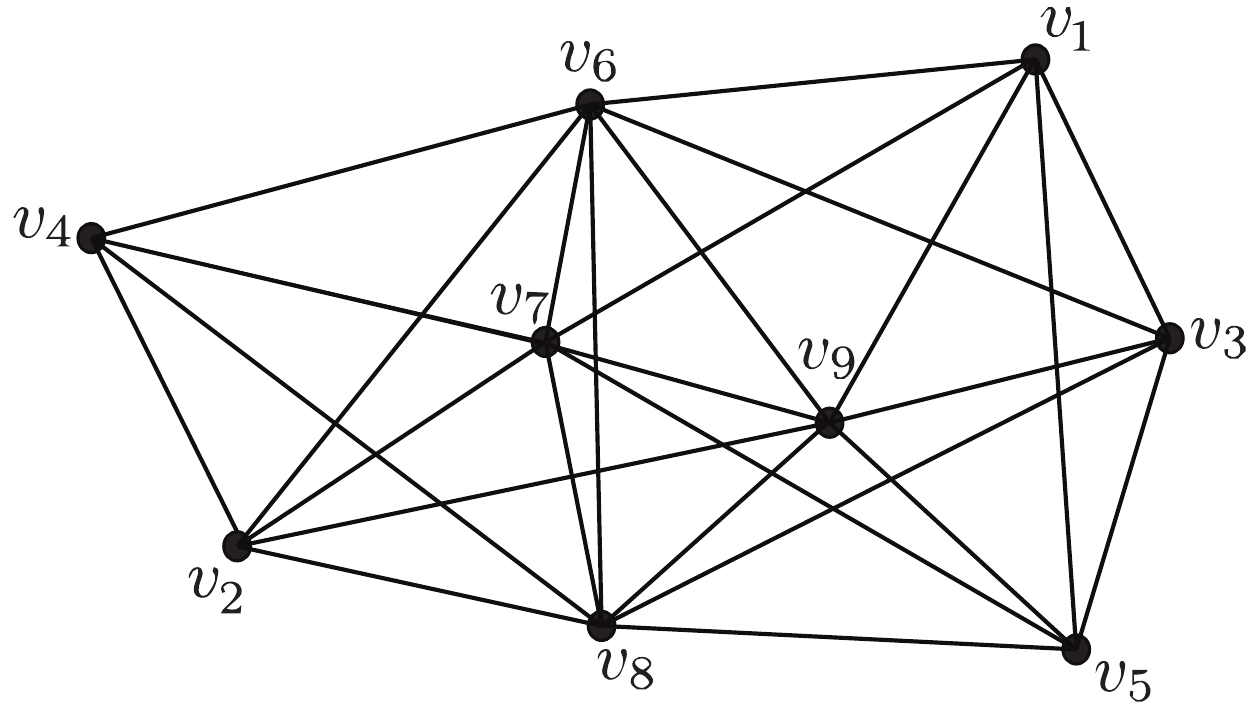}}
\end{picture}
\caption{\small A maxnIL graph with $n=9$ vertices and $4n-10$ edges that is non-apex.}
\label{non-apex}
\end{center}
\end{figure}


\section{The Algorithms Used for the Searches}
\label{CS}
{\it Algorithms for determining if a graph is IL or maxnIL}.
We start with an arbitrary embedding $G \subset S^3$ of the given graph.
Any other embedding $G'$ of the same graph
can be obtained from $G$ by a combination of isotopy
and crossing changes (i.e., allowing edges to ``go through each other'').
A crossing change between a pair of edges
is equivalent to adding a full twist (two half-twists)
between the two edges, as in Figure~\ref{twist}.

\begin{figure}[htpb!]
\begin{center}
\begin{picture}(250, 70)
\put(0,0){\includegraphics[width=3.5in]{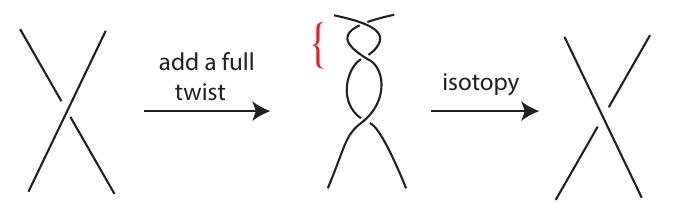}}
\end{picture}
\caption{\small A crossing change between two edges of a graph.}
\label{twist}
\end{center}
\end{figure}

The linking number between any two disjoint cycles in $G'$ can be computed from their corresponding linking number in $G$
by taking into account the number of twists added
between all (disjoint) pairs of edges 
in the two components  of the link in question.

In the algorithm,  for every 2-component link $L$ in the given embedding $G$, we compute its linking number $\lk_G(L)$.
We then assign a variable to every disjoint pair of edges in $G$, representing the number of twists to be added between the two edges in order to obtain a new embedding $G'$.
We write the linking number $\lk_{G'}(L)$ in terms of $\lk_{G}(L)$ and the variables.
Setting all these linking numbers equal to zero gives us a system of linear equations that has an integer solution if and only if there exists an embedding $G'$ of the given graph such that every link in $G'$ has linking number zero.

By the work of Robertson, Seymour, and Thomas~\cite{RST}, a graph has a linkless embedding if and only if it has an embedding with no odd linking numbers.
So, instead of solving the system of linear equations for integer solutions, we work in $\Z_2$ (which is much faster).

Lastly, to determine if a given nIL graph $G$ is maxnIL, we simply check if for every pair of nonadjacent vertices $v, w \in G$ the graph $G + vw$ is IL.
Both of these algorithms were implemented in Mathematica. \\

{\it Searching for $K_6$ minors.}
Recall that $H$ is a minor of $G$ if  deleting and contracting zero or more edges and deleting zero or more vertices in $G$ yields a graph isomorphic to $H$.
However, to check if a connected graph $G$ has a $K_6$ minor, we only need to check if contracting zero or more edges in $G$ yields $K_6$; deleting edges is not necessary since $K_6$ is a complete graph, and deleting vertices is not necessary since $G$ is assumed to be connected.

Now consider a connected graph \(G\) of order \(n > 6.\) Contracting an edge in \(G\) produces a graph with one fewer vertex. So \(G\) has a \(K_6\) minor if and only if there exists a set of \(n-6\) edges in $G$ such that contracting those edges yields a graph isomorphic to \(K_6\).
If $G$ has $m$ edges, there are \({m \choose n-6}\) sets of edges to check.
After contracting the edges in each of these sets, checking if the resulting minor is isomorphic to $K_6$ can be done by simply checking if every vertex has degree~5.
This algorithm was implemented in Python, using the NetworkX library.\\

{\bf Acknowledgement:} This work was made possible in part by a grant of high performance computing resources and technical support from the Alabama Supercomputer Authority.

\bibliographystyle{amsplain}

\begin{thebibliography}{99}

\bibitem{AH1} J. Akiyama and F. Harary. ``A graph and its complement with specified properties I: Connectivity." International Journal of Mathematics and Mathematical Sciences \textbf{2}, no. 2 (1979): 223-228.

\bibitem{AH3} J. Akiyama and F. Harary. ``A graph and its complement with specified properties III: Girth and circumference." International Journal of Mathematics and Mathematical Sciences  \textbf{2}, no. 4 (1979): 685-692.

\bibitem{AH4} J. Akiyama and F. Harary. ``A graph and its complement with specified properties. IV. Counting self-complementary blocks." Journal of Graph Theory  \textbf{5}, no. 1 (1981): 103-107.

\bibitem{A} M. Aires. ``On the number of edges in maximally linkless graphs." Journal of Graph Theory (2021).

\bibitem{BF}  R. Bowen and S. Fisk. ``Generation of triangulations of the sphere." Mathematics of Computation  \textbf{21}, no. 98 (1967): 250-252.

\bibitem{BHK} J. Battle, F. Harary, and Y. Kodama. ``Every planar graph with nine points has a nonplanar complement." Bulletin of the American Mathematical Society  \textbf{68}, no. 6 (1962): 569-571.

\bibitem{CG} J.H. Conway and C. McA. Gordon. ``Knots and links in spatial graphs." Journal of Graph Theory  \textbf{7}, no. 4 (1983): 445-453.

\bibitem{dV} Y.C. de Verdi\`ere. ``Sur un nouvel invariant des graphes et un crit\`ere de planarit\`e." \textit{Journal of Combinatorial Theory, Series B} \textbf{50} (1) (1990) 11--21.


\bibitem{Hog} L. Hogben. ``Nordhaus-Gaddum problems for Colin de Verdiere type parameters, variants of tree-width, and related parameters." Recent Trends in Combinatorics (2016): 275-294.

\bibitem{IM} K. Ichihara and T.W. Mattman. ``Most graphs are knotted."   Journal of Knot Theory and Its Ramifications \textbf{29}, No. 14, 2071003 (2020)

\bibitem{J} L.K. J{\o}rgensen. ``Some Maximal Graphs that are not Contractible to $K_6$." Aalborg Universitetscenter. Institut for Elektroniske Systemer - Rapport, Nr. 1989 : R 89-28 (1989).

\bibitem {Ku} K. Kuratowski. ``Sur le probl\`eme des courbes gauches en topologie." Fundamenta mathematicae \textbf{15}, no. 1 (1930): 271-283.

\bibitem{KLV} A. Kotlov,  L. Lov\'asz, and S. Vempala. ``The Colin de Verdiere number and sphere representations of a graph." Combinatorica \textbf{17}, no. 4 (1997): 483-521.

\bibitem{LS}  L. Lov\'asz and A. Schrijver.  ``A Borsuk theorem for antipodal links and a spectral characterization of linklessly embeddable graphs."  \textit{Proc. Amer. Math. Soc.} \textbf{126}, Number 2 (1998) 1275--1285

\bibitem{Ma} W. Mader. ``Homomorphies\"atze f\"ur graphen." Mathematische Annalen \textbf{178}, no. 2 (1968): 154-168.

\bibitem{M} J. Maharry. ``A splitter for graphs with no Petersen family minor." Journal of Combinatorial Theory, Series B \textbf{72}, no. 1 (1998): 136-139.

\bibitem{MP} B.D. McKay and A. Piperno. ``Practical graph isomorphism, II." Journal of Symbolic Computation \textbf{60} (2014): 94-112.

\bibitem{NPP} R. Naimi, A. Pavelescu, and E. Pavelescu. ``New bounds on maximal linkless graphs." preprint,  arXiv:2007.10522 (2020).

\bibitem{ONPParxiv}  R. Naimi, O. Odeneal, A. Pavelescu, and E. Pavelescu.  ``List of all maxnIL graphs of order up to 11". https://arxiv.org/abs/2108.12946


\bibitem{PP2} A. Pavelescu and E. Pavelescu. ``The complement of a nIL graph with thirteen vertices is IL." Algebraic \& Geometric Topology \textbf{20}, no. 1 (2020): 395-402.

\bibitem{RST} N. Robertson,  P. D. Seymour, and R. Thomas. ``Linkless embeddings of graphs in 3-space." Bulletin of the American Mathematical Society \textbf{28}, no. 1 (1993): 84-89.

\bibitem{Sa} H. Sachs. ``On spatial representations of finite graphs." In Colloq. Math. Soc. Janos Bolyai, vol. 37, pp. 649-662. North-Holland Amsterdam, 1984.

\bibitem{Tutte} W.T. Tutte. ``The non-biplanar character of the complete 9-graph." Canadian Mathematical Bulletin \textbf{6}, no. 3 (1963): 319-330.

\end{thebibliography}

\end{document}